\documentclass[article]{amsart}
\usepackage{cases}
\usepackage{amsmath, amsfonts, pifont, amssymb}
\usepackage{verbatim}
\usepackage[bookmarks=true]{hyperref}
\usepackage{geometry}
\newtheorem{theorem}{Theorem}[section]
\newtheorem{lemma}[theorem]{Lemma}
\newtheorem{proposition}[theorem]{Proposition}

\newcommand{\R}{\mathbb{R}}

\newcommand{\beq}{\begin{equation}}
\newcommand{\eeq}{\end{equation}}
\newcommand{\beqq}{\begin{equation*}}
\newcommand{\eeqq}{\end{equation*}}

\theoremstyle{definition}

\theoremstyle{remark}
\newtheorem{remark}[theorem]{Remark}

\numberwithin{equation}{section}

\numberwithin{equation}{section}

\begin{document}

\address{Chenjie Fan
\newline \indent AMSS, China/University of Chicago, U.S.A.\indent }
\email{fancj@amss.ac.cn}

\address{Zehua Zhao
\newline \indent  University of Maryland, U.S.A. }
\email{zzh@umd.edu}

\title[Decay estimates for nonlinear Schr\"odinger solutions]{Decay estimates for nonlinear Schr\"odinger equations}
\author{Chenjie Fan and Zehua Zhao}
\maketitle

\begin{abstract}
In this short note, we present some decay estimates for nonlinear solutions of 3d quintic, 3d cubic  NLS, and 2d quintic NLS (nonlinear Schr\"odinger equations).
\end{abstract}
\bigskip

\bigskip


\section{Introduction}
\subsection{Statement of main results}
Dispersive estimates play a fundamental role in the study of nonlinear dispersive equations. One typical model of nonlinear dispersive equations is NLS.  It is well known, see for example, \cite{cazenave2003semilinear}, \cite{tao2006nonlinear},  that for the linear propagator $e^{it\Delta}$ of Schr\"odinger equation in $\R^{d}$, one has 
\begin{equation}\label{eq: dispersive}
\|e^{it\Delta}u_{0}\|_{L_{x}^{\infty}}\lesssim \|u_{0}\|_{L_{x}^{1}}t^{-\frac{d}{2}}.
\end{equation}

On the other hand, in the field of defocusing NLS, there are many scattering type results which state that the solution $u$ to the defocusing nonlinear equations, asymptotically behaves like a linear solution. This means one can find $u^{+}$, so that in certain Sobolev space one has 
\begin{equation}
\|u(t)-e^{it\Delta}u^{+}\|_{\dot{H}_{x}^{s}}\rightarrow 0, \text{ as } t\rightarrow \infty.
\end{equation}

One natural question may be that in what sense one can recover estimate \eqref{eq: dispersive} for solutions to the nonlinear equations.

We start with defocusing energy critical NLS in $\R^{3}$ with relatively nice initial data,
\begin{equation}\label{eq: quinticmodel}
\aligned
\begin{cases}
\left(i\partial_t+ \Delta_{\mathbb{R}^{3}} \right) u=   |u|^{4} u, \\
u(0,x) = u_{0}(x) \in H^{3}\cap L^{1},
\end{cases}
\endaligned
\end{equation}
and prove that 
\begin{theorem}\label{main}
Let $u$ solves \eqref{eq: quinticmodel} with initial data $u_{0}$, which is in  $H^{3}\cap L^{1}$.
Then, there exists a constant $C_{u_{0}}$ depending  on $u_{0}$, such that
\begin{equation}
    ||u(t,x)||_{L_x^{\infty}}\leq C_{u_{0}} t^{-\frac{3}{2}}.
\end{equation}
\end{theorem}

\begin{remark}
 One may be able to lower the regularity  of initial data a little bit by sharping the analysis in this article, but the method in this article cannot handle $H^{1}\cap L^{1}$ data.  More importantly, this constant $C_{u_{0}}$, does not only depend on the $L_{x}^{1}$ and $H^{3}$ norm of the data, but also depends on the profile of initial data $u_{0}$. On the other hand, we don't need our initial data in any weighted $L_{x}^{2}$ space.
\end{remark}

\begin{remark}
One may ask similar questions for other NLS models. There is a special case for defocusing mass critical NLS which has a very quick affirmative answer, at least for $d\leq 3$ when the nonlinearity is relatively regular. The observation is: let $u$ be a solution to defocusing mass critical NLS with Schwarz initial data in $\R^{d}$,
\begin{equation}
iu_{t}+\Delta u=|u|^{\frac{4}{d}}u, \quad, u(0,x)=u_{0}.
\end{equation}

 then via pseudo-conformal transformation, one can write $u$ as 
\begin{equation}
u(t,x)=\frac{1}{t^{d/2}}\bar{v}(\frac{1}{t}, \frac{x}{t})e^{i|x|^{2}/4t}
\end{equation}
where $v$ also solves the mass critical NLS global in time, with smooth initial data, and in particular $\|v\|_{L_{x}^{\infty}}\lesssim 1$.
Note that in this case, not only does  $C_{u_{0}}$ exists so that $\
|u(t,x)\|_{L^{\infty}}\leq C_{u_{0}}t^{-3/2}$, but also this $C_{u_{0}}$ only depends on certain norms of $u_{0}$, rather than its profile.
\end{remark}

We also obtain similar results for 3d cubic NLS and 2d quintic NLS.   
Consider 
\begin{equation}\label{eq: cubic}
\aligned
\begin{cases}
\left(i\partial_t+ \Delta_{\mathbb{R}^{3}} \right) u=   |u|^{2} u, \\
u(0,x) = u_{0}(x) \in H^{4}\cap L^{1},
\end{cases}
\endaligned
\end{equation}
we have
\begin{theorem}\label{main2}
Let $u$ solves \eqref{eq: cubic} with initial data $u_{0}$, which is in  $H^{4}\cap L^{1}$,
Then there exists a constant $C_{u_{0}}$, depending $u_{0}$, so that
\begin{equation}
    ||u(t,x)||_{L_x^{\infty}}\leq C_{u_{0}} t^{-\frac{3}{2}}.
\end{equation}
\end{theorem}
\begin{remark}
Again, this $C_{u_{0}}$ depends on the profile of the initial data. \cite{hayashi1986c} also gives decay for this model, with different assumptions. In some sense, they need more integrability in space but less regularity of the data. It should also be noted, by conjugate pseudo-conformal symmetry and energy conservation law, in the same spirit of \cite{hayashi1986c}, one can directly get  time decay for certain $L^{p}$ norm. We thank Jason Murphy for helpful discussion on this point.
\end{remark}
Also consider 
\begin{equation}\label{eq: 2dquintic}
\aligned
\begin{cases}
\left(i\partial_t+ \Delta_{\mathbb{R}^{2}} \right) u=   |u|^{4} u, \\
u(0,x) = u_{0}(x) \in H^{3}\cap L^{1},
\end{cases}
\endaligned
\end{equation}
we have
\begin{theorem}\label{main3}
Let $u$ solves \eqref{eq: 2dquintic} with initial data $u_{0}$, which is in  $H^{3}\cap L^{1}$,
Then there exists a constant $C_{u_{0}}$, depending $u_{0}$, so that
\begin{equation}
    ||u(t,x)||_{L_x^{\infty}}\leq C_{u_{0}} t^{-1}.
\end{equation}
\end{theorem}

\subsection{Background and some further discussions}
Equation \eqref{eq: quinticmodel} is called energy critical because its Hamiltonian/Energy
\begin{equation}
    E(u(t)):=\int \frac{1}{2} |\nabla u(t,x)|^2+\frac{1}{6} |u(t,x)|^6 dx,
\end{equation}
is invariant under the natural scaling. 
The local theory is classical, see, for example, textbooks \cite{cazenave2003semilinear}, \cite{tao2006nonlinear}.

It is indeed highly nontrivial, even for Schwarz initial data, that one can have a global flow for \eqref{eq: quinticmodel}. Nevertheless, it is proven in the seminal work, \cite{colliander2008global}, see also reference therein, that all $H^{1}$ initial data gives a global flow. We summarize their results in the following 
\begin{proposition}[Scattering, global time-space bound and persistence of regularity]
Initial value problem \eqref{eq: quinticmodel} is globally well-posed and scattering. More precisely, for any $u_0$ with finite energy, $u_{0}\in H^{1}$, there exists a
unique global solution $u\in C^0_t({H}^1_x)\cap L^{10}_{x,t}$ such that
\begin{equation}
    \int_{-\infty}^{\infty} \int_{\mathbb{R}^3}|u(t,x)|^{10}dxdt\leq C_{\|u_{0}\|_{H^{1}}},
\end{equation}
for some constant $C(\|u_{0}\|_{H^{1}})$ that depends only on $\|u_{0}\|_{H^{1}}$.
And if $u_0 \in H^s$ for some $s>1$, then $u(t)\in H^s$ for all time $t$, and one has the uniform bounds
\begin{equation}
    \sup\limits_{t\in \mathbb{R}} ||u(t)||_{H^s} \leq C_{\|u_{0}\|_{H^{1}}}||u_0||_{H^s}.
\end{equation}
\end{proposition}
In particular, via this proposition, we may assume in our article,  that there exists $M_1$ such that,
\begin{equation}\label{eq: boundglobal}
   ||u(t)||_{L^{\infty}_tH_x^3} \leq M_1,
\end{equation}
where $u$ solves \eqref{eq: quinticmodel}.

Parallel results hold for equation \eqref{eq: cubic}. If one works on initial data in $H^{1}$, the scattering result is indeed easier.
Much stronger low regularity results holds for equation \eqref{eq: cubic}, \cite{colliander2004global}, \cite{kenig2010scattering}, see also \cite{dodson2020global} and reference therein.

 As a corollary of lower regularity results \cite{colliander2004global}, \cite{kenig2010scattering}, one has 
\begin{proposition}
Initial value problem \eqref{eq: cubic} is globally well-posed and scattering in $H^1$ space. More precisely, for any $u_0$ with finite energy, $u_{0}\in H^{1}$, there exists a
unique global solution $u\in C^0_t({H}^1_x)\cap L^{5}_{x,t}$ such that
\begin{equation}
    \int_{-\infty}^{\infty} \int_{\mathbb{R}^3}|u(t,x)|^{5}dxdt\leq C_{\|u_{0}\|_{H^{1}}},
\end{equation}
for some constant $C(\|u_{0}\|_{H^{1}})$ that depends only on $\|u_{0}\|_{H^{1}}$.
And if $u_0 \in H^s$ for some $s>1$, then $u(t)\in H^s$ for all time $t$, and one has the uniform bounds
\begin{equation}
    \sup\limits_{t\in \mathbb{R}} ||u(t)||_{H^s} \leq C_{\|u_{0}\|_{H^{1}}}||u_0||_{H^s}.
\end{equation}
\end{proposition}
By the above proposition, one may assume 
\begin{equation}\label{eq: boundglobal2}
   ||u(t)||_{L^{\infty}_tH_x^4} \leq M_1.
\end{equation}
where $u$ solves \eqref{eq: cubic}.

For the 2d quintic case, one also has the parallel result, \cite{planchon2009bilinear}\cite{yu2018global},
\begin{proposition}\label{prop: 2dquintic}
Initial value problem \eqref{eq: 2dquintic} is globally well-posed and scattering in $H^1$ space. More precisely, for any $u_0$ with finite energy, $u_{0}\in H^{1}$, there exists a
unique global solution $u\in C^0_t({H}^1_x)\cap L^{8}_{x,t}$ such that
\begin{equation}
    \int_{-\infty}^{\infty} \int_{\mathbb{R}^2}|u(t,x)|^{8}dxdt\leq C_{\|u_{0}\|_{H^{1}}},
\end{equation}
for some constant $C(\|u_{0}\|_{H^{1}})$ that depends only on $\|u_{0}\|_{H^{1}}$.
And if $u_0 \in H^s$ for some $s>1$, then $u(t)\in H^s$ for all time $t$, and one has the uniform bounds
\begin{equation}
    \sup\limits_{t\in \mathbb{R}} ||u(t)||_{H^s} \leq C_{\|u_{0}\|_{H^{1}}}||u_0||_{H^s}.
\end{equation}
\end{proposition}
By the above proposition, one may assume 
\begin{equation}\label{eq: boundglobal3}
   ||u(t)||_{L^{\infty}_tH_x^3} \leq M_1.
\end{equation}
where $u$ solves \eqref{eq: 2dquintic}.

The proof for three different models follows a similar scheme.  Compare to the energy critical model in $\mathbb{R}^{3}$, 3d cubic model needs an extra trick to deal with fact the nonlinearity is not of power high enough, and 2d quintic model needs an (different) extra trick to deal with the fact in dimension 2, the dispersive estimate is not strong enough.

The method in this paper relies on an important observation from \cite{journe1991decay}, that in dimension $d\geq 3$, since the dispersive estimate \eqref{eq: dispersive} gives a decay $t^{-d/2}$ which is integrable in time, then it is possible to treat things in a more perturbative way. However, in this article, we don't necessarily need $d\geq 3$. The reason is that if one wants to view the nonlinear $|u|^{p}u$ as $Vu$, then from the view point of bootstrap or from the results one wants to prove, this $V$ is decaying in time, thus, compensating the non-integrability of $t^{-1}$ when $d=2.$

It turns out Lin and Strauss already studied and obtained the decay of $L^{\infty}$ norm of 3d NLS in their classical paper, \cite{lin1978decay}, and our scheme is similar to theirs. They have different assumptions, but more or less that is due to at that time, the scattering behavior for NLS is not as well understood as today. Their approaches are revisited by a more modern language in the proof of Corollary 3.4 in \cite{grillakis2013pair} for the 3d cubic Hartree equations.\footnote{We get notified of the work \cite{grillakis2013pair} after we finish the current article, which lead us to detail of the work \cite{lin1978decay}. The current article basically illustrates how to get dispersive decay, given one already knows the scattering result. It also covers the 2d quintic case.}

It may be possible to apply vector filed methods and using commutator type estimates to approach those types of problems, see for example, \cite{klainerman1985uniform}. One main difference is it may require the data lives in some weighted $L_{x}^{2}$ space. One difficulty to use $L^{1}\cap H^{s}$ initial data may be the following. Consider for example, energy critical NLS, it will fall into a perturbative regime, in some sense, after evolving time $L$ long enough. However, it is not easy to propagate the $L^{1}$ information from the initial data to time $L$. On the other hand, if one assumes initial data is in some weighted $L^{2}$ space so that $x^{m}D^{n}u_{0}$ is in $L^{2}$, this information is relative easier to be propagated to time $L$. See  \cite{klainerman1983global}, \cite{shatah1982global}. See also, in particular,  \cite{hayashi1986c}, where various decaying estimate are derived for several different  NLS models. As last, we mention that \cite{10.1093/imrn/rnn135} concerns the decay estimate as well. In their paper, they prove global existence of small solutions to NLS by space-time resonance method.

By similar treatments of this article, one can indeed obtain similar results for energy critical NLS in 4d with nice initial data. Also, for in dimension 3 or 4, when the nonlinearity is of higher power and algebraic, one may obtain similar results assuming the critical Sobolev norm stays bounded and associated conditional scattering holds. One may also extend to focusing case, when under certain mass/energy constraint, such a solution is known to scatter. We leave this to interested readers.

\subsection{Notation}
Throughout this note, we use $C$ to denote  the universal constant and $C$ may change line by line. We say $A\lesssim B$, if $A\leq CB$. We say $A\sim B$ if $A\lesssim B$ and $B\lesssim A$. 

Several parameters $M_{1}, M, \delta, L$ will be involved in the analysis. $M_{1}, \|u_{0}\|_{L_{1}}$ is fixed all time. All the parameters $M,\delta, L$ may depend on $M_{1}$, ( and universal constant $C$). Essentially, one first choose $M$, then chooses $\delta$, then chooses $L$. And we only consider $L\geq  100M$.

We also use notation $C_{B}$ to denote a constant depends on $B$.

We use usual $L^{p}$ spaces and Sobolev spaces $H^{s}$.
\subsection{Acknowledgment}
We thank Benjamin Dodson, Zihua Guo, Carlos Kenig and Gigliola Staffilani for helpful comments. C.F. is funded in part by an AMS-Simons Foundation travel grant.
\section{Proof of Theorem \ref{main}}
Let us define 
\begin{equation}
A(\tau):=\sup_{s\leq \tau} s^{3/2}\|u(s)\|_{L_{x}^{\infty}}.
\end{equation}
Note that $A(\tau)$ is monotone increasing.
One aim to prove there is some constant, depending on $u_{0}$, so that, 
\begin{equation}
A(\tau)\leq C_{u_{0}}, \forall t\geq 0.
\end{equation}

Recall we have \eqref{eq: boundglobal}, thus for any given $l$, one can find $C_{l}$ so that,
\begin{equation}
A(\tau)\leq C_{l}, 0\leq \tau\leq l.
\end{equation} 
and the solution is continuous  in time in $L^{\infty}$ since we are working on high regularity data.

Thus, Theorem \ref{main} follows from the following bootstrap lemma.
\begin{lemma}\label{lem: bootstrap}
There exists a constant $C_{u_{0}}$, for that if one has $A(\tau)\leq C_{u_{0}}$, then one has $A(\tau)\leq \frac{C_{u_{0}}}{2}$.
\end{lemma}

Now we turn to the proof of Lemma \ref{lem: bootstrap}. The way to choose $C_{u_{0}}$ will become clear in the proof. Fix $\tau$, we only need to prove for any $t\leq \tau$, one has 
\begin{equation}\label{eq: goal}
\|u(t)\|_{L_{x}^{\infty}}\leq \frac{C_{u_{0}}}{2}t^{-3/2}.
\end{equation}
We recall here, by bootstrap assumption, we can apply the following estimates in the proof
\begin{equation}\label{eq: boosassu}
\|u(t)\|_{L_{x}^{\infty}}\leq C_{u_{0}}t^{-3/2}.
\end{equation}
 Observe, any $\delta$, we can choose $L$, so that for one has 
\begin{equation}\label{eq: scatteringdecay}
\big( \int_{L/2}^{\infty}\|u\|^{10}_{L_{x}^{10}}\big)^{\frac{1}{10}} \leq \delta.
\end{equation}

We will fix two special $\delta, L$ in the proof,  through the exact way choice of those two parameters will only be made clear later.

We will only study $t\geq L$, and estimate all $t\leq L$ directly via
\begin{equation}
\|u(t)\|_{L_{x}^{\infty}}\leq A(L)t^{-3/2}, t\leq L.
\end{equation}

Now we fix $L\leq t\leq \tau$. By Duhamel's Formula, we can write the nonlinear solution $u(t,x)$ as follows,
\begin{equation}
    u(t,x)=e^{it\Delta}u_0+i\int_0^t e^{i(t-s)\Delta}(|u|^4u)(s) ds=u_l+u_{nl}.
\end{equation}

Dispersive estimate gives for some constant $C_{0}$
\begin{equation}\label{eq: linear}
\|u_{l}(t)\|_{L_{t}^{\infty}}\leq C_{0}t^{-3/2}\|u_{0}\|_{L_{x}^{1}}.
\end{equation} 

Now, we need an extra parameter $M$, and we split $u_{nl}$ into 
\begin{equation}
u_{nl}=F_{1}+F_{2}+F_{3},
\end{equation}
where
\begin{equation}
\begin{aligned}
    &F_1(t)=i\int_0^M e^{i(t-s)\Delta}(|u|^4u)(s) ds,\\
    &F_2(t)=i\int_M^{t-M} e^{i(t-s)\Delta}(|u|^4u)(s) ds,\\
    &F_3(t)=i\int_{t-M}^t e^{i(t-s)\Delta}(|u|^4u)(s) ds.\\
    \end{aligned}
\end{equation}
It will be clear how to choose the value of $M$, when we estimate $F_{2}$. And it will be clear how to choose the value of $\delta$, (depending on $M$), when we estimate $F_{3}$. And one chooses $L$,  depending on $\delta$ so that \eqref{eq: scatteringdecay} holds.

We will estimate $F_{1}$ as 
\begin{equation}\label{eq: esf1}
\begin{aligned}
\|F_{1}(t)\|_{L_{x}^{\infty}}&\leq \int_{0}^{M}\|e^{i(t-s)\Delta}|u|^{4}u(s)\|_{L_{x}^{\infty}}\\
&\lesssim M(t-M)^{-3/2}\sup_{s}\|u^{5}(s)\|_{L_{x}^{1}}\\
&\lesssim Mt^{-3/2}\sup_{s}\|u(s)\|_{H^{3}}^{5}\\
&\lesssim MM_{1}^{5}t^{-3/2}.
\end{aligned}
\end{equation}

For $F_2$,  we will apply \eqref{eq: boosassu}, and use pointwise estimate
\begin{equation}\label{eq: pt}
\|e^{i(t-s)\Delta}|u(s)|^{4}u(s)\|_{L_{x}^{\infty}}\lesssim (t-s)^{-3/2}\|u(s)\|_{H^{3}}^{4}\|u(s)\|_{L_{x}^{\infty}}\lesssim C_{u_{0}}M_{1}^{4}(t-s)^{-3/2}s^{-3/2}.
\end{equation}
And one estimate $F_{2}$ via 
\begin{equation}\label{eq: esf2pre}
\|F_{2}(t)\|_{L^{\infty}_x} \leq CC_{u_{0}}M_{1}^{4}\int_{M}^{t-M}(t-s)^{-3/2}s^{-3/2}ds.
\end{equation}
Now, choosing $M$, so that
\begin{equation}\label{eq: choiceofM}
CM_{1}^{4}\int_{M}^{t-M}(t-s)^{-3/2}s^{-3/2}ds\leq \frac{1}{10}t^{-3/2},
\end{equation}
and we can estimate $F_{2}$ as 
\begin{equation}\label{eq: esf2}
\|F_{2}(t)\|_{L^{\infty}_x} \leq \frac{1}{10}C_{u_{0}}t^{-3/2}.
\end{equation}
The estimate of $F_{3}$ will be the most tricky part.  We first state the following technical Lemma,
\begin{lemma}\label{lem: ele}
Let $f$ be a $H^{3}$ function in $\mathbb{R}^{3}$, with
\begin{equation}
\|f\|_{L_{x}^{2}}\leq a, \|f\|_{H^{3}}\leq b,
\end{equation}
then one has 
\begin{equation}
\|f\|_{L^{\infty}}\lesssim (a^{2}b^{3})^{1/5}.
\end{equation}
\end{lemma}
We will present the proof of Lemma \ref{lem: ele} in the Appendix.\vspace{3mm}

Following Lemma \ref{lem: ele}, we estimate the  $L^{2}$-norm and $H^{3}$-norm of $F_{3}$.
Note that $H^{3}$ is a Banach algebra, and $e^{i(t-s)\Delta}$ is unity in $H^{3}$, we directly estimate $\|F_{3}(t)\|_{H^{3}}$ as 
\begin{equation}\label{eq: H3}
\|F_{3}(t)\|_{H^{3}}\leq MM_{1}^{5}.
\end{equation}

For $||F_{3}(t)||_{L^{2}_{x}}$, we will use the fact $t-M\geq L/2$ and rely on \eqref{eq: scatteringdecay}. Also note $t-M\sim t$ since $t\geq L\geq 100M$. We estimate as 
\begin{equation}
\begin{aligned}
\big\|\int_{t-M}^t e^{i(t-s)\Delta}|u|^4uds \big\|_{L^{2}_x} &\leq \int_{t-M}^t || |u|^4u ||_{L_x^2}ds \\
  &\leq  \int_{t-M}^t || u ||^{\frac{15}{8}}_{L_x^{10}}\cdot || u ||^{\frac{5}{2}}_{L_x^{\infty}} \cdot || u ||^{\frac{5}{8}}_{L_x^{2}} ds \\ 
    &\leq CM_1^{\frac{5}{8}}(C_{u_{0}}t^{-\frac{3}{2}})^{\frac{5}{2}} \int_{t-M}^t || u ||^{\frac{15}{8}}_{L_x^{10}} ds \\
     &\leq CM_1^{\frac{5}{8}}(C_{u_{0}}t^{-\frac{3}{2}})^{\frac{5}{2}} \cdot || u ||^{\frac{15}{8}}_{L^{10}_t[t-M,t]L_x^{10}} \cdot M^{\frac{13}{16}} \\
          &\leq CM_1^{\frac{5}{8}}M^{\frac{13}{16}}\delta^{\frac{15}{8}}(C_{u_{0}}t^{-\frac{3}{2}})^{\frac{5}{2}}. \\
\end{aligned}
\end{equation}
Thus, via Lemma \ref{lem: ele}, we derive
\begin{equation}\label{eq: esf3pre}
\aligned
\big\|\int_{t-M}^t e^{i(t-s)\Delta}|u|^4uds \big\|_{L^{\infty}_x} &\leq \big( CM_1^{\frac{5}{8}}M^{\frac{13}{16}}\delta^{\frac{15}{8}}(C_{u_{0}}t^{-\frac{3}{2}})^{\frac{5}{2}} \big)^{\frac{2}{5}} \cdot (M_{1}^{5}M)^{\frac{3}{5}}\\
&\leq C^{\frac{2}{5}}M^{\frac{37}{40}}M_1^{\frac{13}{3}}\delta^{\frac{3}{4}}(C_{u_0}t^{-\frac{3}{2}}).
\endaligned
\end{equation}
Thus, by choosing $\delta$ small enough, according to $M, M_{1}$, we can ensure
\begin{equation}\label{eq: esf3}
\|F_{3}(t)\|_{L_{x}^{\infty}}\leq \frac{1}{10}C_{u_{0}}t^{-3/2}.
\end{equation}
We note that we choose $L$, depending on $\delta$, so that \eqref{eq: scatteringdecay} holds .

We remark here, the choice of $M,L$ does not depend on $C_{u_{0}}$. Indeed, we will choose $C_{u_{0}}$ depending on $M, L$.
 
To summarize, for all $t\leq \tau$, assuming $A(\tau)\leq C_{u_{0}}$, we derive
\begin{itemize}
\item For $t\leq L$, one has 
\begin{equation}
u(t)\leq A(L)t^{-3/2}.
\end{equation}
\item For $L\leq t\leq \tau$, one has, via \eqref{eq: linear}, \eqref{eq: esf1}, \eqref{eq: esf2}, \eqref{eq: esf3}
\begin{equation}
u(t)\leq \{C(\|u_{0}\|_{L_{x}^{1}}+MM_{1}^{5})+\frac{1}{10}C_{u_{0}}+\frac{1}{10}C_{u_{0}}\}t^{-3/2}.
\end{equation}
\end{itemize}

Thus, if one choose 
\begin{equation}
C_{u_{0}}:=10A(L)+C(\|u_{0}\|_{L_{x}^{1}}+MM_{1}^{5}),
\end{equation}
then the desired estimates follows. This ends the proof of Lemma \ref{lem: bootstrap}. Thus proves Theorem \ref{main}.

\section{Proof of Theorem \ref{main2}}
The proof of Theorem \ref{main2} follows the same strategy as Theorem \ref{main}, but needs different idea to treat $F_{3}$. Comparing the two models, the dimensions are the same but the nonlinearities are obviously different. We will briefly overview the setting, discuss the treatment of $F_{3}$ term explicitly and skip other same treatments.\vspace{3mm}

Firstly, by Duhamel's Formula, we write the nonlinear solution $u(t,x)$ to \eqref{eq: cubic} as follows,
\begin{equation}
    u(t,x)=e^{it\Delta}u_0+i\int_0^t e^{i(t-s)\Delta}(|u|^2u)(s) ds=u_l+u_{nl}.
\end{equation}
Clearly dispersive estimate gives for some constant $C_{0}$,
\begin{equation}
\|u_{l}(t)\|_{L_{t}^{\infty}}\leq C_{0}t^{-3/2}\|u_{0}\|_{L_{x}^{1}}.
\end{equation} 
Then, we split $u_{nl}$ into 
\begin{equation}
u_{nl}=F_{1}+F_{2}+F_{3}
\end{equation}
where
\begin{equation}
\begin{aligned}
    &F_1(t)=i\int_0^M e^{i(t-s)\Delta}(|u|^2u)(s) ds,\\
    &F_2(t)=i\int_M^{t-M} e^{i(t-s)\Delta}(|u|^2u)(s) ds,\\
    &F_3(t)=i\int_{t-M}^t e^{i(t-s)\Delta}(|u|^2u)(s) ds.\\
    \end{aligned}
\end{equation}
The process of handling the $F_1$ term and the $F_2$ term are similar to the quintic model case so let's focus on the $F_3$ term.\vspace{3mm} 

Recall, we need to prove, given $M,M_{1}$, there exists $\delta>0$ (small enough and to be decided), such that if $L$
 is chosen, so that,
 \begin{equation}
 \|u\|_{L_{t,x}^{5}[L/2, \infty]}<\delta
 \end{equation}
 then,
 \begin{equation}
 \|F_{3}(t)\|_{L^{\infty}}\leq \frac{C_{u_{0}}}{10}t^{-3/2}.
 \end{equation}
 We need to improve Lemma \ref{lem: ele} to
 \begin{lemma}\label{lem: ele2}
 Let $f(x)$ be a $H^{4}$ function in $\mathbb{R}^{3}$, with  
 \begin{equation}
 \|f\|_{L^{2}}\leq a_{1}, \|\nabla f\|_{L_{x}^{2}}\leq a_{2}, \|f\|_{H^{4}}\leq b. 
 \end{equation}
 Then one has 
 \begin{equation}
\|f\|_{L^{\infty}}\leq  a_{1}^{2/5}a_{2}^{6/25}b^{9/25}.
 \end{equation}
 \end{lemma}
 We will present the proof of Lemma \ref{lem: ele2} in the Appendix.\vspace{3mm}
 
 Now the strategy is to estimate $||F_3||_{H^4}$, $||F_3||_{L_x^2}$ and $||\nabla F_3||_{L_x^2}$ respectively. Then we can apply Lemma \ref{lem: ele2}. Note that $H^{4}$ is a Banach algebra, and $e^{i(t-s)\Delta}$ is unity in $H^{4}$, we directly estimate $\|F_{3}(t)\|_{H^{4}}$ as 
\begin{equation}\label{eq: H4}
\|F_{3}(t)\|_{H^{4}}\leq MM_{1}^{3}.
\end{equation}
Then we turn to the estimate for $||F_3||_{L_x^2}$. 
 \begin{equation}
\aligned
\big\|\int_{t-M}^t e^{i(t-s)\Delta}|u|^2uds \big\|_{L^{2}_x} &\leq \int_{t-M}^t || |u|^2u ||_{L_x^2}ds \\
  &\leq  \int_{t-M}^t || u ||^{\frac{7}{6}}_{L_x^{5}}\cdot || u ||^{\frac{13}{10}}_{L_x^{\infty}} \cdot || u ||^{\frac{8}{15}}_{L_x^{2}} ds \\ 
    &\leq CM_1^{\frac{8}{15}}(C_{u_{0}}t^{-\frac{3}{2}})^{\frac{13}{10}} \int_{t-M}^t || u ||^{\frac{7}{6}}_{L_x^{5}} ds \\
     &\leq CM_1^{\frac{8}{15}}(C_{u_{0}}t^{-\frac{3}{2}})^{\frac{13}{10}} \cdot || u ||^{\frac{7}{6}}_{L^{5}_t[t-M,t]L_x^{5}} \cdot M^{\frac{23}{30}} \\
          &\leq  CM_1^{\frac{8}{15}}M^{\frac{23}{30}}\delta^{\frac{7}{6}}(C_{u_{0}}t^{-\frac{3}{2}})^{\frac{13}{10}}. \\
\endaligned
\end{equation}
Also, we can deal with the estimate for $||\nabla F_3||_{L_x^2}$ as follows.
  \begin{equation}
\aligned
\big\|\int_{t-M}^t e^{i(t-s)\Delta}\nabla(|u|^2u)ds \big\|_{L^{2}_x} &\leq \int_{t-M}^t || \nabla(|u|^2u) ||_{L_x^2}ds \\
  &\leq  \int_{t-M}^t || \nabla u ||_{L_x^{2}}\cdot || u ||^{2}_{L_x^{\infty}}  ds \\ 
    &\leq MM_1\cdot (C_{u_0}t^{\frac{3}{2}})^2.
\endaligned
\end{equation}
At last, putting the above estimates together and applying Lemma \ref{lem: ele2}, we have,
\begin{equation}
\aligned
\big\|\int_{t-M}^t e^{i(t-s)\Delta}|u|^2uds \big\|_{L^{\infty}_x} &\leq \big( CM_1^{\frac{8}{15}}M^{\frac{23}{30}}\delta^{\frac{7}{6}}(C_{u_{0}}t^{-\frac{3}{2}})^{\frac{13}{10}} \big)^{\frac{2}{5}} \cdot (MM_1\cdot (C_{u_0}t^{\frac{3}{2}})^2)^{\frac{6}{25}} \cdot (M_{1}^{3}M)^{\frac{9}{25}} \\
&\leq C^{\frac{2}{5}}M_1^{\frac{23}{25}}M^{\frac{68}{75}}\delta^{\frac{7}{15}}(C_{u_0}t^{\frac{3}{2}}). 
\endaligned
\end{equation}
Thus, by choosing $\delta$ small enough, according to $M, M_{1}$, we can ensure
\begin{equation}
\|F_{3}(t)\|_{L_{x}^{\infty}}\leq \frac{1}{10}C_{u_{0}}t^{-3/2},
\end{equation}
as desired.

\section{Proof of Theorem \ref{main3}}
The proof of Theorem \ref{main3} follows the same strategy as Theorem \ref{main} and Theorem \ref{main2}, but needs different idea to treat $F_{2}$ (since the decay is not enough for the 2d case). We will briefly overview the 2d setting, discuss the treatment of the $F_{2}$ term in detail and skip other same treatments.\vspace{3mm}

Let us define 
\begin{equation}
A(\tau):=\sup_{s\leq \tau} s\|u(s)\|_{L_{x}^{\infty}}.
\end{equation}
Note that $A(\tau)$ is monotone increasing.
One aim to prove there is some constant, depending on $u_{0}$, so that, 
\begin{equation}
A(\tau)\leq C_{u_{0}}, \forall t\geq 0.
\end{equation}
Similar to the previous two cases, Theorem \ref{main3} follows from the following bootstrap lemma.
\begin{lemma}\label{lem: bootstrap3}
There exists a constant $C_{u_{0}}$, for that if one has $A(\tau)\leq C_{u_{0}}$, then one has $A(\tau)\leq \frac{C_{u_{0}}}{2}$.
\end{lemma}

Now we turn to the proof of Lemma \ref{lem: bootstrap3}.

Firstly, by Duhamel's Formula, we write the nonlinear solution $u(t,x)$ to \eqref{eq: 2dquintic} as follows,
\begin{equation}
    u(t,x)=e^{it\Delta}u_0+i\int_0^t e^{i(t-s)\Delta}(|u|^4u)(s) ds=u_l+u_{nl}.
\end{equation}
Clearly dispersive estimate gives for some constant $C_{0}$,
\begin{equation}
\|u_{l}(t)\|_{L_{t}^{\infty}}\leq C_{0}t^{-1}\|u_{0}\|_{L_{x}^{1}}.
\end{equation} 
Then, we split $u_{nl}$ into 
\begin{equation}
u_{nl}=F_{1}+F_{2}+F_{3}
\end{equation}
where
\begin{equation}
\begin{aligned}
    &F_1(t)=i\int_0^M e^{i(t-s)\Delta}(|u|^4u)(s) ds,\\
    &F_2(t)=i\int_M^{t-M} e^{i(t-s)\Delta}(|u|^4u)(s) ds,\\
    &F_3(t)=i\int_{t-M}^t e^{i(t-s)\Delta}(|u|^4u)(s) ds.\\
    \end{aligned}
\end{equation}
The process of handling the $F_1$ term and the $F_3$ term are quite similar to the 3d quintic model so let's focus on the $F_2$ term. We recall the following fact according to Proposition \ref{prop: 2dquintic}.
\begin{equation}
    \int_0^{+\infty} \|u(s)\|^{8}_{L^8_{x}} ds<\infty.
\end{equation}
We start with a pointwise estimate applying the bootstrap assumption, together with the H\"older and the Sobolev inequality as follows 
\begin{equation}\label{eq: pt3}
\|e^{i(t-s)\Delta}|u(s)|^{4}u(s)\|_{L_{x}^{\infty}}\lesssim (t-s)^{-1}\|u(s)\|_{H^{3}}^{3} \|u(s)\|_{L_{x}^{8}} \|u(s)\|_{L_{x}^{\infty}}\lesssim C_{u_{0}}M_{1}^{3}(t-s)^{-1}s^{-1}\|u(s)\|_{L_{x}^{8}}.
\end{equation}
And one estimate $F_{2}$ via 
\begin{equation}\label{eq: esf2pre3}
\aligned
\|F_{2}(t)\|_{L^{\infty}_x} &\leq CC_{u_{0}}M_{1}^{3}\int_{M}^{t-M}(t-s)^{-1}s^{-1} \|u(s)\|_{L_{x}^{8}} ds\\
&\leq CC_{u_{0}}M_{1}^{3}\int_{M}^{\frac{t}{2}}(t-s)^{-1}s^{-1} \|u(s)\|_{L_{x}^{8}} ds\\
&+CC_{u_{0}}M_{1}^{3}\int_{\frac{t}{2}}^{t-M}(t-s)^{-1}s^{-1} \|u(s)\|_{L_{x}^{8}} ds \\
&\leq 2CC_{u_{0}}M_{1}^{3} t^{-1}\int_{M}^{\frac{t}{2}}s^{-1} \|u(s)\|_{L_{x}^{8}} ds \\
&+ 2CC_{u_{0}}M_{1}^{3} t^{-1} \int_{\frac{t}{2}}^{t-M}(t-s)^{-1} \|u(s)\|_{L_{x}^{8}} ds \\
&\leq 2CC_{u_{0}}M_{1}^{3} t^{-1}\big(\int_{M}^{\frac{t}{2}}s^{-\frac{8}{7}}ds\big)^{\frac{7}{8}} \cdot \|u(s)\|_{L_{t,x}^{8}}  \\
&+ 2CC_{u_{0}}M_{1}^{3} t^{-1} \big(\int_{\frac{t}{2}}^{t-M}(t-s)^{-\frac{8}{7}} ds\big)^{\frac{7}{8}} \cdot \|u(s)\|_{L_{t,x}^{8}}. \\
\endaligned
\end{equation}
Now, choosing $M$, so that
\begin{equation}\label{eq: choiceofM3}
4CC_{u_{0}}M_{1}^{3} t^{-1} \big(\int_{M}^{\infty}s^{-\frac{8}{7}} ds\big)^{\frac{7}{8}} \cdot \|u(s)\|_{L_{t,x}^{8}} \leq \frac{1}{10}C_{u_{0}}t^{-1}.
\end{equation}
Thus eventually we can estimate $F_{2}$ as
\begin{equation}
\|F_{2}(t)\|_{L^{\infty}_x} \leq \frac{1}{10}C_{u_{0}}t^{-1}.
\end{equation}

 \section{Appendix}
We record the proof for Lemma \ref{lem: ele} and Lemma \ref{lem: ele2} in this section.

\begin{proof}[Proof of Lemma \ref{lem: ele}]
Indeed, assume $|f(x_{0})|=c> 0$, for $|x-x_{0}|\sim \frac{1}{100}\frac{c}{b}$, one still has $|f(x)|\sim c$ since $\|\nabla f(x)\|_{L^{\infty}}\leq \|f\|_{H^{3}}\leq b$. Thus, one has 
\begin{equation}
c^{2}\frac{c^{3}}{b^{3}}\lesssim \|f\|_{L_{x}^{2}}^{2}\leq a^{2}.
\end{equation}
And the desired estimate follows.
\end{proof}
\begin{proof}[Proof of Lemma \ref{lem: ele2}]
Apply Lemma \ref{lem: ele}, one derive
\begin{equation}
\|\nabla f\|_{L^{\infty}}\leq a_{2}^{2/5}b^{3/5}.
\end{equation}
Now, following the proof of Lemma \ref{lem: ele2}, one derives
\begin{equation}
\|f\|_{L^{\infty}}\lesssim a_{1}^{2/5}(a_{2}^{2/5}b^{3/5})^{3/5}=a_{1}^{2/5}a_{2}^{6/25}b^{9/25}.
\end{equation}
\end{proof}

\bibliographystyle{amsplain}
\bibliographystyle{plain}
\bibliography{BG}

\end{document}